\documentclass[12pt]{article}
\usepackage{latexsym}
\usepackage{a4wide}

\newtheorem{theorem}{Theorem}
\newtheorem{lemma}{Lemma}[section]

\newenvironment{proof}{\rmfamily\upshape\mdseries\small{\noindent\normalsize\bfseries Proof:}}{\nopagebreak\rule{2mm}{2mm}\newline}

\newcommand{\Magma}{{\scshape Magma}}

\def\C{{\bf C}}

\def\N{{\bf N}}
\def\R{{\bf R}}
\def\eps{\varepsilon}

\def\JT{\texttt{JordanTest}}
\def\PT{\texttt{PrimitiveTest}}
\def\AT{\texttt{AltsymTest}}

\begin{document}
\title{Fast detection of giant permutation groups}
\author{W. R. Unger\thanks{The author acknowledges the support of ARC grant DP160104626} \\
School of Mathematics and Statistics, \\
University of Sydney, \\
Sydney, Australia}
\maketitle

\begin{abstract}
We give a 1-sided randomised algorithm to detect when a permutation group
of degree $n$, given by generators, contains the alternating group $A_n$.
This improves on standard methods, and on an algorithm of
P. Cameron and J. Cannon.
\end{abstract}

\section{Introduction}\label{sec:0}

The \emph{giant} permutation groups are the alternating and symmetric groups.
When we are given permutations generating a group, we need to determine quickly
whether or not they generate a giant, as they are so much larger than other
permutation groups of the same degree that we need to treat them specially.
Standard texts, such as \cite{S03} and \cite{HEO}, offer the following solution.
\begin{theorem}\label{text}
\begin{enumerate}
\item
Let $G$ be a transitive permutation group of degree $n$, and let $p$ be a
prime with $n/2 < p < n-2$. If $G$ contains an element of order
divisible by $p$ then $A_n \le G$.
\item
The proportion of elements of $A_n$ and $S_n$ that have order divisible by
some prime $p$ where $n/2 < p < n-2$, is asymptotic to
$\frac{\log 2}{\log n}$ .
\end{enumerate}
\end{theorem}

Determining whether or not $G$ is transitive from the generators is a standard
test, so we will now assume that $G$ is transitive.

We use this theorem to test for a giant by generating some random
elements of G and testing the element orders, looking for a large prime
factor.  If such a prime turns up, then return that $G$ is a giant.
If not, then return that $G$ is not a giant.

Note that the answer ``giant'' is always correct, while ``not a giant''
could be incorrect. This is a one-sided Monte Carlo test.

For a true Monte Carlo algorithm, there is a question to answer:
how many elements do we test?
The number depends on how tolerant we feel about the possibility of
``$G$ is not a giant'' being incorrect.
We will use $\eps$ to denote the probability of error we will allow.
Assuming that each random permutation is chosen independently of any other,
the second part of the theorem lets us choose the number of elements to test.
If $n = 10^6$. then $\log 2/\log n$ is about $0.05$.
Testing 45 elements gives error probability $\eps < 0.1$,
90 elements gives $\eps < 0.01$. These numbers increase linearly with
$\log n$ and with $-\log\eps$.

The time taken is dominated by the time to generate random
elements of $G$ and to determine cycle lengths and the element order.
We want a method that uses few elements.

We say that $x \in S_n$ is a \emph{Jordan element} whenever,
for all $G \le S_n$, $G$ primitive and $x \in G$ implies $A_n \le G$. 

We use the name Jordan in honour of C. Jordan's work of the 1870s, where
he showed that 2-cycles, 3-cycles, and $p$-cycles, with $p$ prime and 
$p < n-2$, are Jordan elements.

Jordan's result on $p$-cycles has been used for fast detection of giant
permutation groups for many years. In the late 1950s,
Parker \& Nikolai \cite{PN58} used Jordan's result as a test,
knowing that the groups under consideration were primitive.
Cannon \cite{C84} describes a more general method, where the group is first
tested for primitivity, and, if so, Jordan's result is used.

It is unclear to the author when the method given above,
i.e. to attempt to prove both primitivity and the Jordan property from
random elements, emerged.
The permutation group programs of Cayley and \Magma\ contained such an
algorithm that performs better than the Theorem~\ref{text} method.
This algorithm is due to P. Cameron and J. Cannon and was alluded to in
\cite{CC}, but no details were published.
As $n\to\infty$, the Cameron-Cannon method has to test at least 6
(for $\eps < 0.1$) and 11 ($\eps < 0.01$) elements, and experiments did not
rule out needing increasing numbers with degree.
Experimental results were 6 and 12 elements at degree $10^6$,
very slowly increasing with degree. See section \ref{CCalg} for more
discussion of this algorithm and the derivation of these lower bounds.

In this article, we improve Cameron and Cannon's algorithm, again
using cycle structure of random elements to prove the group
primitive, and to prove that the primitive group contains the alternating
group. This problem is closely related to the question of invariable
generation of the symmetric group, see \cite{D92, LP}, and has applications
to detecting when the Galois group of a polynomial is a giant.

\section{Jordan elements}

Since the 1870s there has been considerable history of proving that certain 
permutations are Jordan elements.
As the Jordan property is invariant under conjugation in $S_n$, it is
the cycle structure of elements that determines the property.

The next result is Corollary 1.3 of \cite{J14}. This is a
version of Jordan's result, following the classification of finite
simple groups, which removes the primality condition.
\begin{theorem} \label{j14}
If $x \in S_n$ is an $\ell$-cycle, with $1 < \ell < n-2$, then
$x$ is a Jordan element.
\end{theorem}

The following collects results ranging from Jordan in the 1870s to Saxl 1977. 
\begin{theorem}[Jordan, Manning, Weiss, Saxl] \label{jmws}
If $x \in S_n$ has prime order $p$, with cycle type $m$ $p$-cycles and
$k$ fixed points, where $m < p$, then any of the following implies $x$
is a Jordan element.
\begin{itemize}
\item $m = 2$: $k > 3$, or $p > 3$ and $k > 2$;
\item $m = 3$: $k > 3$;
\item $m = 4$: $k > 5$, or $p > 5$ and $k > 4$;
\item $m = 5$: $k > 2$;
\item $m = 6$: $k > 7$, or $p > 7$ and $k > 6$;
\item $m = 7$: $k > 8$;
\end{itemize}
\end{theorem}
The results for $m \le 4$ are proved in \cite{M09}, $m = 5$ in \cite{S77},
and $m = 6,7$ in \cite{W28}.
Results of this kind, for $m\le5$, were known to Jordan in 1873
(see \cite{M09}).

For larger $m$ we have the following.
\begin{theorem}[Manning \cite{M18}] \label{m18}
If $x \in S_n$ has prime order $p$, with cycle type $m$ $p$-cycles and
$k$ fixed points, where $m > 5$, $p > 2m-2$ and $k > 4m-4$,
then $x$ is a Jordan element.
\end{theorem}
\begin{theorem}[Praeger \cite{P79}] \label{p79}
If $x \in S_n$ has prime order $p$, with cycle type $m$ $p$-cycles and
$k$ fixed points, where $1 < m < p$, $k > 5m/2 - 2$, $n\ne 9$ and
$n\ne{c \choose 2}$, where $c = m + (p+1)/2$,
then $x$ is a Jordan element.
\end{theorem}

If necessary, stronger results could be derived from \cite{LS}.
This work gives a classification of all primitive groups containing an element
of prime order $p$, with cycle type $m$ $p$-cycles and $k$ fixed points,
where $m < p$.

We note a simple fact.
\begin{lemma}
If $x \in S_n$ and some power of $x$ is a Jordan element, then
$x$ is a Jordan element.
\end{lemma}

Consider an algorithm, \JT, which takes as input the cycle type of
a permutation $x$, and returns true if some power of $x$
has cycle type as listed in one of
Theorems~\ref{j14}, \ref{jmws}, \ref{m18}, or \ref{p79},
and returns false otherwise.
When \JT\ returns true, it has shown that $x$ is a Jordan element; 
if false is returned, then we draw no conclusion.

\section{Properties of Random Permutations}
Suppose that $P$ is a property of elements of $S_n$,
and let $p_n$ denote the proportion of elements of $S_n$ satisfying $P$.
We say that $P$ is true for \emph{almost all} permutations when
$\lim_{n\to\infty} p_n = 1$.
If $q_n$ is the proportion of even permutations in $S_n$ satisfying $P$,
and $\lim_{n\to\infty} q_n = 1$, then we say that $P$ is true for almost
all even permutations. The following lemmas are easily proved.
\begin{lemma}\label{lem_alt}
If $P$ is true for almost all permutations, then $P$ is
true for almost all even permutations.
\end{lemma}
\begin{lemma}\label{lem_aaand}
If both $P$ and $Q$ are true for almost all permutations, then
$(P \mbox{ and } Q)$ is true for almost all permutations.
\end{lemma}

It has long been known that almost all permutations are Jordan elements.
We now collect results on properties of random permutations, and prove that
\JT\ detects almost all permutations as being Jordan elements.

We consider a random $x\in S_n$, chosen from the uniform distribution.
Let the number of cycles of $x$ be $L$.
A number of properties of the random variable $L$ are well-known.
It is elementary to show that mean and variance are both asymptotic to
$\log n$, see for instance \cite[\S1.2.10]{Knu}.
(All logarithms in this work are to base $e$.)
This is enough to justify the following theorem.
\begin{theorem}
Fix $\delta > 0$. For almost all permutations,
$$ (1-\delta)\log n < L < (1+\delta)\log n. $$
\end{theorem}
\begin{proof}
It follows from Lemma~3 of \cite{D92} and Stirling's approximation that
the proportion of permutations of degree $n$ with number of cycles outside
the range $(1-\delta)\log n < L < (1+\delta)\log n $ is
$O(n^{\delta-(1+\delta)\log(1+\delta)})$, so going to 0 as a power of $n$.
\end{proof}
In fact, it is known that the distribution of $L$ approaches a
normal distribution \cite{G42, FS}.
This power of $n$, for small $\delta$, is $-\delta^2/2+O(\delta^3)$.
For future reference, the proportion of permutations with more than
$2\log n$ cycles (i.e. $\delta =1$) is $O(n^{-0.38})$.


The next result is from \cite{ET}.
\begin{theorem}
\label{et}
Let $\omega:\N\to\R$ be some function with
$\lim_{n\to\infty} \omega(n) = \infty$.
Then almost all permutations have order divisible by some prime 
greater than $n/\exp\left(\omega(n)\sqrt{\log n}\right)$.
\end{theorem}

Suppose that the cycles of $x\in S_n$ are $\C_1, \C_2, \dots, \C_L$, where
the cycles are ordered by the least point in each cycle. Denote the length
of cycle $\C_i$ by $C_i$. The following is proved in \cite{LP},
as Claims 1 \& 2.
\begin{theorem}
\label{lp}
Let $x\in S_n$ have $L$ cycles, and cycle lengths
$C_1, C_2,\dots, C_L$ as above. Put $m = \lfloor\sqrt{\log n}\rfloor$.
Then, for almost all $x$ we have $L > m$, and
\begin{enumerate}
\item For $1\le i < j \le L$, $\gcd(C_i, C_j) < n^{0.9}$.
\item For $1\le i\le m$, $C_i > n^{0.99}$;
\item $\gcd(C_1,C_2,\dots,C_m) = 1$;
\end{enumerate}
\end{theorem}

We can now prove an asymptotic result about \JT.
\begin{theorem}
For almost all permutations $x$, the \JT\ algorithm applied to
the cycle type of $x$ will return true.
\end{theorem}
\begin{proof}
Put $\omega(n) = 0.3\sqrt{\log n}$ in Theorem~\ref{et}.
Consider an $x\in S_n$ such that the properties of Theorems~\ref{et}
and \ref{lp} hold for $x$. Almost all $x$ have these properties.
We now show that the cycle type of $x$ is recognised by \JT\ as that
of a Jordan element.

Let $p$ be the largest prime divisor of the order of $x$,
so $p > n^{0.7}$ by Theorem~\ref{et}.
Also, by Theorem~\ref{lp}, there is some cycle of $x$ with length
$\ell > n^{0.99}$ and $\ell$ not divisible by $p$.

Let $y$ be a power of $x$ so that $y$ has order $p$.
The cycle type of $y$ is $m$ $p$-cycles and $k$ fixed points,
where $m < n^{0.3}$ and $k \ge \ell > n^{0.99}$.
We may assume that $n$ is large enough so that $n^{0.7} > 2n^{0.3}$,
which implies $p > 2m$, and $n^{0.99} > 4n^{0.3}$, so $k > 4m$.
This cycle type is recognised by \JT\ as a Jordan element using
one of Theorems~\ref{j14}, \ref{jmws}, or \ref{m18}.
\end{proof}

\section{Proving Primitivity}

We consider the possibility that the transitive group $G\le S_n$ is
imprimitive, having a non-trivial block system consisting of $r$ blocks
each of size $s$, so $rs = n$.

Such a block system imposes constraints on the cycle type of elements
of $G$. These constraints are generally not satisfied by random elements of
$S_n$.
We may be able to prove $G$ primitive by showing that no $r$ is possible given
the cycle type of an element of $G$.
We will see that this is so for almost all permutations.

Suppose $x\in G$, $G\le S_n$ is transitive, and $r$ divides $n$.
How can we determine that $G$ has no block system with $r$ blocks?

Suppose that there is such a block system, and that $x$ has a cycle of
length $\ell$. 
Then this cycle induces a cycle of blocks with length $c\ge 1$.  We then have:

\begin{enumerate}
\item $c$ divides $\ell$, $\ell\le cs$ and $c \le r$;
\item There are other cycles of $x$, with lengths multiples of $c$, and lengths
adding to $cs-\ell$. 
\end{enumerate}

Deciding whether or not item 2 holds is an instance of a subset sum
decision problem.
The subset sum decision problem is NP-complete (see \cite{gj79})
so instances may be hard to solve.
We will limit the subset sum problems we will consider in size so that an
exponential algorithm to solve them will run in $o(n)$ time.

We use the above as follows: For each $r$ dividing $n$, if there is some
cycle of length $\ell$ with no possible corresponding $c$ because the above
criteria cannot be met, then a block system with $r$ blocks can be eliminated
as a possibility. Once all $r$ have been eliminated, the group has been
proved primitive.

We define an algorithm \PT\ to take as input the cycle type of an element
$x\in S_n$.

For each $r$ dividing $n$ with $1 < r < n$, and for each cycle length $\ell$,
it attempts to eliminate $r$ as above.
If all $r$ are eliminated, then true is returned, otherwise false is returned.
If true is returned by \PT\ applied to some cycle type, then
any transitive group containing an element of that cycle type is primitive.

\begin{theorem} \label{pt}
For almost all permutations $x$, the \emph{\PT} algorithm applied to the
cycle type of $x$ will return true.
\end{theorem}
\begin{proof}
This proof follows a portion of the proof of \cite{LP} Theorem~1,
where it is shown that almost all permutations have cycle structure such that
the cycle structure cannot belong to an element of a transitive and
imprimitive group. Set $r_n = \exp(\log\log n\sqrt{\log n})$.
Observe that $r_n=o(n^\delta)$ for all $\delta > 0$.
Possible $r$ dividing $n$ are split into 3 cases:

\noindent
\textbf{Case 1:} $1 < r \le r_n$. \newline
We may assume that the cycle structure satisfies the properties
given in Theorem~\ref{lp}. For each $r$ in this range, there exists a cycle
\C, with length $\ell$ not a multiple of $r$.
Assume there is a block system with $r$ blocks, and let $B$ be the union of
all blocks that intersect \C. Then $|B| < n$, since $r$ does not divide $\ell$,
and $|B| = an/r$ for some integer $a$, $1 \le a < r \le r_n$.
It is shown in \cite{LP}, proof of Theorem~1, Case~1, that the proportion of
permutations that support such an invariant subset is $o(n^{-0.009})$,
so almost all permutations do not support such an invariant subset.
Whether or not such a subset is supported by a given cycle structure will
be detected by solving the subset sum decision problem given above,
so almost all permutations will allow \PT\ to eliminate all $r$ in this range.

\noindent
\textbf{Case 2:} $r_n < r \le n/r_n$. \newline
By Theorem~\ref{et} with $\omega(n) = \log\log n$, we may assume that
the permutation has order divisible by a prime $p > n/r_n$,
so $p > r$ and $p > s$.
If $\ell$ is a cycle length divisible by $p$, then there is no $c$ which
simultaneously satisfies: $c$ divides $\ell$, $c\le r$, and $\ell/c\le s$.
This cycle length eliminates all $r$ in this range.

\noindent
\textbf{Case 3:} $n/r_n < r < n$. \newline
We may assume that the cycle structure satisfies the properties
given in Theorem~\ref{lp}. For each $r$ in this range, there exists a cycle
length $\ell$ with $\ell > n^{0.99}$ and $\ell$ not a multiple of $s$.
Any $c$ must satisfy, for sufficiently large $n$, 
$c \ge \ell/s > \ell/r_n > n^{0.99}/r_n > n^{0.9}$.
As $s$ does not divide $\ell$, we have $cs > \ell$, so for the subset sum
problem to have a solution, there must be another cycle with length a
multiple of $c > n^{0.9}$.
However such a cycle contradicts the properties given in Theorem~\ref{lp}.
So, in this case, for large $n$ and with a cycle structure satisfying the
properties of Theorem~\ref{lp}, the subset sum decision problem will eliminate
all $r > n/r_n$, as there will be no other cycle with length a multiple of $c$.
\end{proof}

\section{The full \AT\ algorithm}

Here we describe the full algorithm for detecting a giant permutation group.
An assumption in this description is the ability to generate a
random group element cheaply, when only generators of the group are given.
For a discussion of this matter see section 3.2.2 of \cite{HEO}.

\begin{figure}[htbp]
\noindent
The \AT\ algorithm.

\noindent
Input: Generators for $G\le S_n$, and an integer $k \ge 0$.

\noindent
Output: If true is returned then $A_n\le G$.

\medskip
\begin{tabbing}
xxxx\=xxxx\=xxxx\=xxxx\=xxxx\=\kill
start \AT;\\
\>if G is not transitive then return false; end if;\\
\>jordan := false;\\
\>r\_list := list of non-trivial divisors of $n$;\\
\>for i := 1 to k do\\
\>\>if jordan and \#r\_list eq 0 then\\
\>\>\>return true;\\
\>\>end if;\\
\>\>x := Random(G);\\
\>\>csx := CycleStructure(x);\\
\>\>if number of cycles in csx $< 2\log n$ then\\
\>\>\>if not jordan then\\
\>\>\>\>jordan := \JT(csx);\\
\>\>\> end if;\\
\>\>\>r\_list := [r in r\_list $\vert$ r not eliminated as in \PT(csx)];\\
\>\>end if;\\
\>end for;\\
\>return (jordan and \#r\_list eq 0);\\
end \AT;\\
\end{tabbing}
\end{figure}

While the \JT\ algorithm is called directly by \AT, \PT\ is not called
directly, but the mechanisms of \PT\ are applied to a sequence of up to
$k$ cycle structures in order to eliminate possible block systems.

For any transitive group $G$, and any $k\ge 1$, if \AT\ applied to $(G,k)$
returns true, then $A_n \le G$. As with the standard algorithm, it is possible
to have $A_n \le G$ and for \AT\ to return false.
The probability of this happening is the error probability, which can be
reduced by increasing $k$.
We have seen that almost all permutations have fewer than $2\log n$ cycles,
and \AT\ treats a permutation with more than $2\log n$ cycles as evidence
against the group given being a giant.
Using the previous results for \JT\ and \PT, we have:
\begin{theorem}
\begin{enumerate}
\item
For all $\eps>0$ and $N>0$ there exists $k= k(N,\eps)$ such that
if the input to \emph{\AT} is $(G,k)$, with $n \ge N$ and $A_n\le G$,
then the probability that \emph{\AT} returns false is $<\eps$.
\item
For fixed $\eps>0$, we may take $k(N,\eps)$ to be a non-increasing
function of $N$, with limit 1 as $N\to\infty$.
\end{enumerate}
\end{theorem}

We now consider the time taken for the algorithms when degree is $n$ and
there are $L$ cycles in the permutation. In the following we assume
classical algorithms for integer arithmetic, factorization
of integers $\le n$ in time $O(n^\delta)$ for all $\delta>0$,
and use of the algorithm of \cite{HS} for the subset sum decision problem.
Theorem~315 of \cite{HW} tell us that the number of divisors of $n$ is
$O(n^\delta)$ for all $\delta>0$, so we may compute a list of all
divisors of any $\ell \le n$ in time $O(n^\delta)$ for all $\delta>0$.

With the above assumptions, we can implement \JT\ to run in time
$$O(L^2(\log n)^3 + Ln^\delta(\log n)^2),$$ and
\PT\ to run in time $$O(Ln^\delta \log n(L \log n + 2^{L/2})),$$ 
for all $\delta > 0$. The component $O(2^{L/2}\log n)$ is the time
to solve a subset sum decision problem. 

Turning to \AT: In calls to \JT\ and \PT\ we have
$L < 2\log n$, so we find that calls to these functions
use time $O(kn^{0.8})$. Time to generate $k$ random permutations and
determine cycle type is at least $Akn$, for some constant $A$,
which dominates the times for \JT\ and \PT\ operations.

The asymptotic results given above rely on theorems that do not apply
at the currently feasible degrees for permutation group computation.
The performance at computable degree was checked by experiments.

\begin{enumerate}
\item
The time to generate random permutations and work out cycle type
greatly exceeds the time for \JT\ and \PT\ operations. 
\item
The dominant time within  \JT\ and \PT\ at low degree is
factorising cycle lengths and getting the list of divisors. 
\item 
The time for the subset sum problem seems negligible.
\item 
In summary, $k(N,\eps)$ is the key timing parameter.
\end{enumerate}

Experiments with the algorithm suggest the values for
$k(N,\eps)$ given in Table 1.

\begin{table}[htbp]
\begin{center}\begin{tabular}{r||c|c|c|c}
 &\multicolumn{4}{c}{$\eps$} \\
 $N$ & $0.1$ & $0.05$ & $0.02$ & $0.01$ \\ \hline
 $10$ & 6 & 6 & 8 & 9 \\
 $20$ & 4 & 5 & 6 & 6 \\
 $30$ & 4 & 4 & 5 & 6 \\
 $50$ & 3 & 3 & 4 & 5 \\
 $100$ & 3 & 3 & 4 & 4 \\
 $1000$ & 2 & 2 & 3 & 3 \\
 $10^4$ & 2 & 2 & 2 & 2 \\
 $10^5$ & 1 & 2 & 2 & 2 \\
 $10^6$ & 1 & 1 & 2 & 2 \\
 $10^8$ & 1 & 1 & 1 & 2 \\
 $10^{12}$ & 1 & 1 & 1 & 1 \\
\end{tabular}\end{center}
\caption{Experimental values for $k(N,\eps)$.}
\end{table}

\section{Other matters}
\subsection{The alternating group}
The asymptotic results given here apply equally to detecting the
alternating and symmetric groups, see Lemma~\ref{lem_alt}.
At very low degrees, less than 20 say,
the \AT\ algorithm detects the symmetric group more quickly than it detects
the alternating group, mainly due to a scarcity of detectable Jordan elements.
This effect quickly becomes insignificant as the degree rises.

\subsection{The influence of divisors of the degree}
The factorization and list of divisors of the degree is used in the
primitivity test.
At one end of the scale, when $n$ is prime, the test is over once
primality is established.

There is a very obvious difference in performance of
\PT\ between odd and even degrees.
Even degrees need distinctly larger $k$ for the same $\eps$.
The experimental values given are for even degrees.
There are also detectable differences when the degree is divisible by 3 and 4,
but the effects are much smaller.

\subsection{On the Cameron-Cannon algorithm}\label{CCalg}
The Cameron-Cannon algorithm mentioned earlier is similar to  
the \AT\ algorithm described here, but without the use of the subset sum
criterion\footnote{In early versions of \Magma\ the implementation was
slightly weaker than this, in that $r$'s were eliminated sequentially,
leading to cycle length information being wasted.}.
Examining the proof of Theorem~\ref{pt}, we expect it to need more
random permutations to eliminate small and large~$r$.

In particular, when $n$ is even, to eliminate $r=2$ the
Cameron-Cannon algorithm must find an element having a cycle of odd length
$> n/2$. The limiting proportion of such
elements in $S_n$ and $A_n$ is $\frac12\log 2$, slightly over 1 in 3.
This leads to the lower bounds mentioned earlier.
In experiments this case was the case that determined the number
of random elements needed by this algorithm.

In contrast, the \AT\ algorithm eliminates $r=2$ by finding
an element which has no fixed set of size $n/2$ and does not have
all cycles of even length.
The proportion of permutations with all cycles of even length is
asymptotic to $\sqrt{\frac{2}{\pi n}} = O(n^{-0.5})$
(by Lemma~4 of \cite{D92} and Stirling's approximation).
The main lemma of \cite{LP} shows that the proportion of elements
that have fixed set of size $n/2$ is $O(n^{-0.01})$.
This exponent is probably rather pessimistic:
the final discussion of \cite{LP} suggests that the same methods could
give $O(n^{-0.3})$.

\subsection{Constructing random cycle types}\label{rand_cyc}
As the majority of time in the \AT\ algorithm is taken up generating
random permutations and getting their cycle type, a lot of time is saved 
when conducting the experiments needed to determine $k(N,\eps)$ by just
constructing cycle types. This also allows us to extend the degrees considered
beyond the degrees where we are willing to write down a permutation. 

Lemma~1 of \cite{D92} may be used to generate random cycle types of
elements of $S_n$, with distribution corresponding to the uniform 
distribution on elements of $S_n$.
This was used in the experiments for Table 1, which only required
cycle types of elements of $A_n$ and $S_n$.

\subsection{On Monte Carlo-ness}
The \AT\ algorithm as stated here is a randomised algorithm, but not a
Monte Carlo algorithm, as I have not given a method for taking an error
probability $\eps>0$ and processing this to find the $k$ parameter.

The experiments do give usable values for
$k$ that are much less than other methods commonly used.
We may extend the range of $\eps$ by using a geometric approximation
to get an upper bound on $k$ giving smaller $\eps$ than those in the table.
For instance, doubling $k$ will extend from $\eps = 0.01$ to $\eps = 0.0001$
even if such $k$ are a little larger than needed.

\subsection{Implementation}
The algorithm described here is implemented as part of \Magma's
routines for permutation groups.
All the results of Theorems~\ref{j14}, \ref{jmws}, \ref{m18},
and \ref{p79} are used to check for
Jordan elements, as well as some special cases at degrees $\le 8$.
The result of Praeger, Theorem
\ref{p79}, which was not used in the
asymptotic results, is useful in the range of degrees where computation
is practical.
Within the \AT\ algorithm, the time for \PT\ dominates the time for \JT,
and using results of \cite{LS} to improve the recognition of Jordan elements
has not seemed worthwhile.

Random elements are generated using the product-replacement algorithm.
It appears to be worthwhile to use \AT\ on the random elements generated
before the mixing time for product-replacement is reached.

\bibliographystyle{plain}
\bibliography{is_altsym}

\end{document}